\documentclass[11pt,a4paper,twosides]{amsart}
\usepackage{amssymb,amsmath,amsthm,mathrsfs,graphicx,xcolor}
\usepackage{amsrefs}
\usepackage{leftidx}
\usepackage{hyperref}

\theoremstyle{plain}
\newtheorem{theorem}{Theorem}[section]

\newtheorem{lemma}[theorem]{Lemma}

\newtheorem{proposition}[theorem]{Proposition}
\theoremstyle{remark}
\newtheorem{remark}[theorem]{Remark}

\allowdisplaybreaks[4]

\numberwithin{equation}{section}

\newcommand{\C}{\mathbb{C}}
\newcommand{\R}{\mathbb{R}}

\newcommand{\F}{\mathcal{F}}

\renewcommand{\Im}{\operatorname{Im}}
\renewcommand{\Re}{\operatorname{Re}}

\newcommand{\Jbr}[1]{\left\langle #1 \right\rangle}

\def\({\left(}
\def\){\right)}
\def\<{\left\langle}
\def\>{\right\rangle}
\def\le{\leqslant}
\def\ge{\geqslant}

\def \pa{\partial}

\def \F{\mathcal{F}}

\newcommand{\eps}{\varepsilon}

\DeclareMathOperator{\sn}{sn}
\DeclareMathOperator{\cn}{cn}
\DeclareMathOperator{\dn}{dn}

\DeclareMathOperator{\am}{am}

\newcommand{\todayd}{\the\year/\the\month/\the\day}

\theoremstyle{definition}

\renewcommand{\phi}{\varphi}
\renewcommand{\epsilon}{\varepsilon}
\renewcommand{\tilde}{\widetilde}

\begin{document}
\title[Quadratic NLS system under mass-resonance in 2D]
{Modified scattering type asymptotic behavior\\ for a quadratic nonlinear Schr\"odinger system\\ under the mass-resonance condition in two dimensions}

\author{Satoshi Masaki}
\address{Department of Mathematics, 
Faculty of Science, Hokkaido University,
Kita 10, Nishi 8, Kita-Ku, Sapporo, Hokkaido, 060-0810, Japan}
\email{masaki@math.sci.hokudai.ac.jp}


\author{Kota Uriya}
\address{Department of Applied Mathematics, Faculty of Science, 
Okayama University of Science, Okayama, 700-0005, Japan}
\email{uriya@ous.ac.jp}


\maketitle

\begin{abstract}
We study a system of nonlinear Schr\"odinger equations under the mass-resonance condition and provide a complete description of the asymptotic behavior of small solutions in two space dimensions.
Our analysis is based on a detailed study of an associated system of ordinary differential equations governing the asymptotic profile. We establish a phase-amplitude representation for this ODE with arbitrary initial data, where the amplitude is expressed explicitly in terms of Jacobi elliptic functions and the phase is given by elliptic integrals of the third kind.
As a consequence, we obtain a fully explicit characterization of the asymptotic profile for the original PDE. In particular, the long-time behavior is described by a modified-scattering-type profile whose profile evolves according to the above integrable structure.
While elliptic-function-type asymptotics were previously constructed for special solutions in final-state problems, the present work provides the first complete characterization,
in two space dimensions, of asymptotic dynamics associated with arbitrarily small initial data through elliptic-function-type profiles.
\end{abstract}

\section{Introduction}

\subsection{Background and problem}

In this paper, we study the long-time behavior of solutions to quadratic nonlinear Schr\"odinger systems in two spatial dimensions. 
A fundamental feature of such systems is that their asymptotic behavior is strongly influenced by resonance structures in the nonlinear interaction.

As a primary example, we consider the system
\begin{equation}\label{E:NLS}
	\left\{
	\begin{aligned}
	&i \partial_t u_1 + \frac1{2m_1}\partial_x^2 u_1 = u_2^2, \\
	&i \partial_t u_2 + \frac1{2m_2} \partial_x^2 u_2 = \overline{u_2} u_1,\\
	\end{aligned}
	\right.
\end{equation}
for $(t,x)\in \R \times \R^2$, where $m_1,m_2>0$ under the initial condition
\begin{equation}\label{E:IC}
	u_j(0)= u_{j,0}  \in H^{0,2}(\R^2),
\end{equation}
where $H^{0,2}(\R^2)$ is the weighted $L^2$ space. 
It is well known that the qualitative behavior of \eqref{E:NLS} depends crucially on the mass ratio. 
In particular, the mass-resonance condition
\begin{equation}\label{E:MR}
	m_1 = 2m_2
\end{equation}
aligns the nonlinear interaction with the characteristic phase
of the linear evolution.
As a consequence, the interaction becomes phase-resonant,
producing long-range effects in the asymptotics.

The system \eqref{E:NLS} arises from physical models
and has been studied extensively in various spatial dimensions.
It possesses the symmetry that if $(u_1(t),u_2(t))$
is a solution, then so is
$(e^{2i\theta}u_1(t),e^{i\theta}u_2(t))$
for any $\theta \in \R$.
Under the resonance condition \eqref{E:MR},
this symmetry becomes compatible with the dispersive phase
of the linear evolution; in particular,
\eqref{E:MR} is the necessary and sufficient condition
for Galilean invariance.
For related results on \eqref{E:NLS},
see
\cite{GMXZ,HaMa,NoPa,OU2013,IKN2019,IKN2020,IOU,HOT,HIN,Ham2018}
and references therein.

In two spatial dimensions, quadratic nonlinearities are critical, and under the resonance condition \eqref{E:MR}, nonlinear effects may persist in the large-time behavior. 
Such long-range effects is already observed for single equations with gauge-invariant nonlinearities, where logarithmic phase corrections appear (see, e.g., \cite{HaNa,Oz,GiOz}).
For systems such as \eqref{E:NLS}, the situation is more intricate due to the interaction between components.

A major development in the analysis of such systems is due to Katayama--Sakoda, who established a general framework reducing the asymptotic behavior of solutions to that of an associated reduced ODE system. 
In the case of \eqref{E:NLS} under the resonance condition \eqref{E:MR}, the profile is governed by
\begin{equation}\label{E:ODE}
	\left\{
	\begin{aligned}
	&i \partial_\tau A_1 = A_2^2, \\
	&i \partial_\tau A_2 = \overline{A_2} A_1.
	\end{aligned}
	\right.
\end{equation}
Thus, the asymptotic analysis naturally splits into a PDE part, which is treated in a unified manner, and an ODE part describing the effective dynamics.

While the integrability of such reduced systems is classical in the physics literature,
particularly in the context of nonlinear optics and three-wave interactions
(see, e.g., \cite{ABDP}),
the emphasis there has been primarily on the evolution of intensities
and energy transfer between waves.
In those works, the dynamics is typically described in terms of
real-valued quantities such as $|A_j|^2$, together with quadratic conserved quantities of Manley--Rowe type, and explicit formulas are obtained for these amplitudes
in terms of elliptic functions.

By contrast, the full complex-valued structure of the system,
in particular the explicit representation of the phase,
has not been pursued in detail.
Accordingly, the emphasis was naturally placed on physically observable quantities.
From the viewpoint of dispersive PDEs, however,
the phase plays a crucial role in determining the asymptotic behavior,
and a complete description of the solution requires
an explicit representation of both amplitude and phase.

\subsection{Integrable structure and its limitations}

The system \eqref{E:ODE} belongs to a class of integrable dynamical systems closely related to models studied in the physics literature, in particular in the context of resonant wave interactions \cite{ABDP}. 
These systems possess conserved quantities and admit reductions that reveal an underlying integrable structure.

A notable example is the three-wave interaction model
\begin{equation}\label{E:3ODE}
	\left\{
	\begin{aligned}
	&i \partial_\tau A_1 = A_2 A_3, \\
	&i \partial_\tau A_2 = \overline{A_3} A_1,\\
	&i \partial_\tau A_3 = \overline{A_2} A_1,
	\end{aligned}
	\right.
\end{equation}
which arises from the resonant system
\begin{equation}\label{E:3NLS}
	\left\{
	\begin{aligned}
	&i \partial_t u_1 + \frac1{2m_1}\partial_x^2 u_1 = u_2 u_3, \\
	&i \partial_t u_2 + \frac1{2m_2}\partial_x^2 u_2 = \overline{u_3} u_1,\\
	&i \partial_t u_3 + \frac1{2m_3}\partial_x^2 u_3 = \overline{u_2} u_1,
	\end{aligned}
	\right.
\end{equation}
where $m_j \neq 0$ ($j=1,2,3$) satisfy the resonance condition  $m_1=m_2+ m_3$.

The system \eqref{E:3ODE} admits multiple conserved quantities, such as
\[
	|A_1|^2 + |A_2|^2, \qquad |A_2|^2 - |A_3|^2,
\]
and can be regarded as a natural extension of \eqref{E:ODE}. 
Indeed, the latter is the special case $|A_2|^2-|A_3|^2=0$ of \eqref{E:3ODE}.

The integrability of systems such as \eqref{E:ODE} and \eqref{E:3ODE} has been essentially known in the physics literature (\cite{ABDP}). 
In particular, their amplitude dynamics can be described in terms of elliptic functions, and various special solutions have been constructed.
However, from the viewpoint of the Cauchy problem for the original PDEs, this information is not sufficient. 
Existing results typically do not provide a closed-form representation of general complex-valued solutions. 
In particular, the phase dynamics is not explicitly described, and the solutions are not given in a form that can be directly used to characterize asymptotics for arbitrary initial data.

The purpose of this paper is to provide a complete description of solutions to the reduced systems \eqref{E:ODE} and \eqref{E:3ODE} for arbitrary initial data, in a form suitable for applications to the corresponding PDEs.
This result bridges the gap between the integrable structure known in the physics literature and the requirements of the PDE analysis. 
In particular, it provides the missing component in the Katayama--Sakoda framework, allowing one to characterize the asymptotic behavior of solutions to \eqref{E:NLS} and \eqref{E:3NLS} for arbitrary small initial data.

A notable feature of the resulting asymptotics is that the nonlinear interaction affects not only the phase but also the amplitude, leading to a stronger long-range effect than the modified scattering observed in single equations.
In particular, the asymptotic profiles are described
through elliptic functions.
Such phenomena have previously been observed
for certain one-dimensional cubic nonlinear Schr\"odinger systems
\cite{Ma2,Ma3,Ur},
as well as in the final state problem
for two-dimensional quadratic nonlinear Schr\"odinger systems
\cite{HLN2011,HLN2016,OU2013,OU2015,Ur16}.
To the best of our knowledge, however,
the present work provides 
the first example of elliptic-function asymptotics
in the initial value problem
for a two-dimensional quadratic nonlinear Schr\"odinger system.

\subsection{Main results}

In this paper, we provide such a complete description.
We derive explicit formulas for both the amplitude and the phase
of solutions to the reduced ODE system,
where the amplitude is expressed in terms of Jacobi elliptic functions
and the phase in terms of elliptic integrals of the third kind.
The same structural feature appears in the three-wave interaction model,
where the integrability and amplitude dynamics have long been understood,
but a complete complex-valued description remains absent.

Our main result is an explicit phase-amplitude representation of solutions. 
More precisely, we show that the amplitude can be described in terms of Jacobi elliptic functions, while the phase is given explicitly by elliptic integrals of the third kind. 
This yields an explicit representation of the full complex-valued solution.

%
%
%

\subsubsection{Elliptic function and elliptic integral}

To describe the solution explicitly, we introduce standard elliptic functions.
For $k\in [0,1]$ and $t\in \R$, we define the amplitude function $\am (t;k)$ with parameter $k$ by the relation
\[
	t=\int_0^{\am(t;k)} \frac{d\theta}{\sqrt{1-k^2 \sin^2 \theta}}.
\]
Then, we define Jacobi elliptic function $\sn(t;k)$ with parameter $k$ by
\[
	\sn(t;k) = \sin \am (t;k).
\]
For $n\le 1$, $\phi\in (-\frac\pi2,\frac\pi2)$, and $k\in[0,1]$, we define the elliptic integral of the third kind $\Pi(n; \phi ;k)$ by
\[
	\Pi(n;\phi ;k) = \int_0^\phi \frac{d\theta}{(1-n\sin^2 \theta)\sqrt{1-k^2\sin^2\theta}}.
\]
This restriction ensures that the integrand remains non-singular.
In the expressions below, the argument $\phi$ will be given by $\am(t;k)$, 
for which the integral is well-defined for the parameters appearing in our problem.
Moreover, by the change of variables $\theta=\am(s;k)$, we obtain
\begin{equation}\label{E:Pirelation}
	\Pi(n;\am(t;k) ; k) = \int_0^t \frac{ds}{1-n \sn^2 (s;k)}.
\end{equation}
This expression is well-defined on $\R$ if $n<1$; it also remains well-defined in the case $k=1$ due to the restricted range of $\am(t;1)$.

\subsubsection{Expression of the solution to \eqref{E:3ODE}}
Let us turn to the description of a non-trivial solution of \eqref{E:3ODE}.
We denote $\rho_j = |A_j|^2$. First observe that 
\[
	 \rho  = \rho_1(0) + \tfrac12\rho_2(0)+ \tfrac12\rho_3(0)>0
\]
is conserved. 
Using symmetries of the system, we may reduce the problem to a normalized setting.
Note that \eqref{E:3ODE} has the following three symmetries:
If $(A_1(\tau),A_2(\tau),A_3(\tau))$ is a solution to \eqref{E:3ODE} then the following three triplets are also solutions:
\begin{itemize}
\item 
The scaled triplet $(\lambda A_1(\lambda \tau),\lambda A_2(\lambda \tau),\lambda A_3(\lambda \tau))$, where $\lambda>0$  (scale invariance).
\item The triplet
$(A_1(\tau),A_3(\tau),A_2(\tau))$ 
\item  The triplet $(e^{i\theta_2+i\theta_3}A_1(\tau),e^{i\theta_2}A_2(\tau),e^{i\theta_3}A_3(\tau))$, where $\theta_2,\theta_3 \in \R$ (gauge invariance).
\end{itemize}
By utilizing these invariance, we may suppose that 
\begin{equation}\label{E:inv}
	\rho=1, \quad \rho_2 \ge \rho_3, \quad A_2(0)\in \R_{\ge0}, \quad A_3(0) \in \R_{\ge0}
\end{equation}
without loss of generality.
In what follows, we suppose this assumption.

To parametrize the solution, we introduce two conserved quantities:
\[
    \gamma := \Re (\overline{A_1} A_2 A_3), \quad c =  \rho_1 + \rho_3.
\]
By \eqref{E:inv}, one has $0\le c \le 1$.
Define
\[
   m= m(c) := 
    \tfrac{2}{27} (1+ \sqrt{3(1-c)^2+1}) (2-\sqrt{3(1-c)^2+1})^2 ,
\]
which is the maximum of the function $  [0,c ]\ni e \mapsto  e(2-c- e )(c-e) .$
Note that $[0,1] \ni c \mapsto m(c) \in [0,\frac4{27}]$ is strictly monotone increasing
and hence a bijection.
The following relation is valid:
\begin{equation}\label{E:gm}
    0\le \gamma^2 \le m \le \tfrac4{27}.
\end{equation}
Note that $m=\frac4{27}$ ($\Leftrightarrow c=1$) if and only if $\rho_2=\rho_3$.
Define a function
\[
    f(e) := e(2-c - e)(c - e) - \gamma^2.
\]
Let $e_1 \le e_2 \le e_3$ be three real roots of the equation $f(e)  = 0$.
One sees that
\[
   0\le  e_1 \le \tfrac{2-\sqrt{3(1-c)^2+1}}{3} \le e_2 \le c \le 1\le  2-c \le e_3.
\]
We remark that 
$\gamma^2 = m$ corresponds the case $e_1=e_2$.

\begin{theorem}\label{T:sol_formula}
	Under the assumption \eqref{E:inv}, the solution to \eqref{E:3ODE} is written as follows:
	\begin{itemize}
	\item If $\gamma^2=m$, i.e., if $\gamma= \sigma \sqrt{m}$, $\sigma \in \{\pm1\}$, then
	\begin{align*}
	A_1(\tau) &= \sigma \sqrt{e_1 }  e^{-i\frac{\gamma}{e_1}  \tau}, \\
	A_2(\tau) &= \sqrt{2-c-e_1} e^{-i\frac{\gamma}{2-c-e_1}  \tau}, \\
	A_3(\tau) &= \sqrt{c-e_1} e^{-i\frac{\gamma}{c-e_1} \tau}
	\end{align*}
    for $\tau\in \R$, where $e_1=\tfrac{2- \sqrt{3(1-c)^2+1}}{3}$ in this case.    
	\item  If $\gamma=0$ then the solution is either the equilibrium $(A_1(0),0,0)$ or
    there exists $\tau_0$ such that
    	\begin{align*}
	A_1(\tau) &= -i\sqrt{c} \sn (\sqrt{2-c}(\tau-\tau_0);\sqrt{\tfrac{c}{2-c}}) , \\
	A_2(\tau) &= \sqrt{2-c}\dn (\sqrt{2-c}(\tau-\tau_0);\sqrt{\tfrac{c}{2-c}}), \\
	A_3(\tau) &= \sqrt{c} \cn (\sqrt{2-c}(\tau-\tau_0);\sqrt{\tfrac{c}{2-c}})
	\end{align*}
	for $\tau\in \R$
    \item If $0<\gamma^2 <m$ then we have 
    \[
    	0< e_1 < e_2 <  c \le 2-c <e_3
    \]
     and the representation
    \[
    A_j(\tau) = \sqrt{\rho_j(\tau)} e^{ i\Phi_j(\tau)}
    \]
    for $\tau\in \R$,
    where 
    \begin{align*}
    \rho_1(\tau) &=  e_1(1   
    - n_1 \sn^2( \omega
    (\tau-\tau_0); k )),\\
    \Phi_1(\tau) &=  - \tfrac{\gamma}{e_1\omega} \Pi(n_1; \am \left(\omega(\tau-\tau_0); k\right); k) + \arccos \tfrac{\gamma}{\sqrt{f(\rho_1(0))+\gamma^2}} ,\\
    n_1 &= -\tfrac{e_2-e_1}{e_1}<0,
    \end{align*}
    \begin{align*}
    \rho_2(\tau) &= (2-c-e_1)(1-n_2 \sn^2( \omega
    (\tau-\tau_0); k )),\\
    \Phi_2(\tau) &= - \tfrac{\gamma}{(2-c-e_1)\omega} \Pi(n_2; \am (\omega(\tau-\tau_0);k); k),\\
    n_2&= \tfrac{e_2-e_1}{2-c-e_1}<1,
    \end{align*}
    and
    \begin{align*}
    \rho_3(\tau) &=  (c-e_1)(1   
    - n_3 \sn^2( \omega
    (\tau-\tau_0); k )),\\
    \Phi_3(\tau) &= - \tfrac{\gamma}{(c-e_1)\omega} \Pi(n_3; \am \left(\omega(\tau-\tau_0); k\right); k) ,\\
    n_3&= \tfrac{e_2-e_1}{c-e_1}<1,
    \end{align*}
with the parameters
\[
    \omega = \sqrt{e_3-e_1}>0, \quad k = \sqrt{\frac{e_2-e_1}{e_3-e_1}}\in (0,1)
\]
and suitable constants $\tau_0\in\R$  and $\theta_1$ given by the initial data.
\end{itemize}
\end{theorem}

\subsubsection{Expression of the solution to \eqref{E:ODE}}

Let us turn to the description of the solution to \eqref{E:ODE}.
Note that it enjoys the following symmetries:
If $(A_1(\tau),A_2(\tau))$ is a solution to \eqref{E:3ODE} then the following two pairs are also solutions:
\begin{itemize}
\item 
The scaled pair $(\lambda A_1(\lambda \tau),\lambda A_2(\lambda \tau))$, where $\lambda>0$  (scale invariance).
\item  The pair $(e^{2i\theta_2}A_1(\tau),e^{i\theta_2}A_2(\tau))$, where $\theta_2 \in \R$ (gauge invariance).
\end{itemize}
By utilizing these invariance, we may suppose without loss of generality that 
\begin{equation}\label{E:inv2}
	|A_1|^2+|A_2|^2=1,  \quad A_2(0)\in \R_{\ge0}.
\end{equation}
The description is an immediate consequence of Theorem \ref{T:sol_formula}.
Indeed, the following theorem shows that 
the system \eqref{E:ODE} is embedded into \eqref{E:3ODE}
through the reduction $A_2 = A_3$.
\begin{theorem}\label{T:connection}
The pair $(A_1,A_2)$ is a solution to \eqref{E:ODE} if and only if the triplet $(A_1,A_2,A_2)$ is a solution to \eqref{E:3ODE}.
Consequently, any nontrivial solution to \eqref{E:ODE} satisfying \eqref{E:inv2}
is given as in Theorem \ref{T:sol_formula} with $c=1$ or equivalently with $m=\frac{4}{27}$.
\end{theorem}

\subsection{Outline of the method}

%

Let us briefly comment on the mechanism behind the phase representation in Theorem \ref{T:sol_formula}.
A key observation is that the evolution of the phase can be expressed
explicitly in terms of quadratic quantities of the solution,
which themselves admit a closed description in terms of elliptic functions,
together with the conserved quantities of the system.
As a consequence, the phase is obtained by a single integration step,
leading naturally to elliptic integrals.

The use of such quadratic quantities as an intermediate object
was introduced in the first author's previous work \cite{Ma2},
where the nonlinearity is cubic and the associated quadratic quantities
satisfy a closed system of equations in a direct manner.
In the present case, the nonlinearity is quadratic,
and the corresponding quantities do not form a closed system at first glance.
However, by combining them with the conserved quantities,
we show that the system can still be reduced to a closed form.

We note that, at the level of amplitude dynamics,
this strategy is consistent with classical treatments in the physics literature \cite{ABDP},
where similar reductions appear implicitly,
although the full complex-valued reconstruction is not carried out.

\subsection{Application to NLS systems}

We turn to the application of our analysis on \eqref{E:ODE} and \eqref{E:3ODE} to the corresponding NLS systems \eqref{E:NLS} and \eqref{E:3NLS}.
Since \eqref{E:NLS} is realized as the special case $m_2=m_3$ and $u_{2,0}=u_{3,0}$ of \eqref{E:3NLS}, we state the result for \eqref{E:3NLS}.

%
\begin{theorem}
\label{T:NLS}
Suppose $m_j \neq0 $ ($j=1,2,3$) satisfies $m_1=m_2+m_3$.
There exist  $\nu>0$ and $\varepsilon_0>0$ such that the following holds:
For any initial datum $(u_{1,0},u_{2,0},u_{3,0}) \in H^{0,2}(\R^2; \C^3)$
satisfying $\varepsilon := \sum_{j=1}^3 \| u_{j,0} \|_{H^{0,2}} \le \varepsilon_0$,
there exists a unique global solution  $u=(u_1,u_2,u_3) \in C \cap L^\infty (\R; L^2)$ to \eqref{E:3NLS} under $u_{j}(0) = u_{j,0}$ ($j=1,2,3$)
 such that
$$e^{-i\frac{t}{2m_j}\Delta}u_j(t) \in C \cap L^\infty(\R; H^{0,2}).$$
Moreover, there exists a function $\eta_+=(\eta_{+,1},\eta_{+,2},\eta_{+,3})\in (C\cap L^\infty \cap L^2) (\R^2; \C^3)$
such that for $1\le j\le 3$ and $2 \le p \le \infty$,
\begin{align}
\label{E:asymp}
 u_j(t,x)=(it)^{-1}e^{im_j\frac{|x|^2}{2t}} A_j\left(\log t; \frac{x}{t}\right)+O_{L^p}(\varepsilon t^{-\frac{3}{2}+\frac1{p}+\nu \eps^2})
\end{align}
as $t\to\infty$, where $A=(A_1,A_2,A_3)$ 
denotes the solution to \eqref{E:3ODE} with initial condition $A(0;\xi)=\eta_+(\xi)$.
There also exists $\eta_-\in (C\cap L^\infty \cap L^2) (\R^2; \C^3)$
such that the asymptotics \eqref{E:asymp} holds as $t\to-\infty$.

If $m_2=m_3$ and $u_{2,0}=u_{3,0}$ hold in addition, $u_2$ and $u_3$ coincide for all time and $(u_1,u_2)$ is a unique global solution to \eqref{E:NLS}. Further, $\eta_+$ satisfies
$\eta_{+,2}=\eta_{+,3}$. Further, $A_2$ and $A_3$ coincide and  $(A_1,A_2)$ is a solution to \eqref{E:ODE}.
The same is true for $\eta_{-}$.
\end{theorem}

%
%

%

It follows by combining the above ODE analysis with the result of Katayama and the first author \cite{KaMa}, which generalizes the framework of Katayama--Sakoda \cite{KaSa}.

The assumptions required in these results concern the structure of the nonlinearity.
More precisely, one assumes gauge invariance, which in the present setting is ensured by the resonance condition $m_1 = m_2 + m_3$, together with the so-called bounded weak null condition.
The latter requires that there exists a constant $B \ge 1$ such that
\[
	\sup_{\tau \in \R} |A(\tau)| \le B |A(0)|
\]
for all solutions to the associated reduced ODE system.
For \eqref{E:3ODE}, this condition is satisfied as a consequence of the conservation of $\rho$.

The general theory in \cite{KaMa} is designed to treat a broad class of systems, including those with derivative nonlinearities.
In particular, the bounded weak null condition is strictly weaker than the existence of conserved quantities (for the corresponding ODE system).
In the present setting, however, \eqref{E:3NLS} admits classical quadratic conservation laws,
and therefore the assumption on the initial data can be reduced.
For the sake of completeness, we briefly outline the proof in Section \ref{S:NLS}.

\subsection{Finite time blowup for an ODE related to \eqref{E:ODE} and \eqref{E:3ODE}}\label{SS:blowup}

The following generalized model is also considered in the literature (see e.g. \cite{HLN2012,HLN2016}):
\begin{equation}\label{E:gNLS}
	\left\{
	\begin{aligned}
	&i \partial_t u_1 + \frac1{2m_1}\partial_x^2 u_1 = u_2^2, \quad (t,x) \in \R \times \R^2, \\
	&i \partial_t u_2 + \frac1{2m_2} \partial_x^2 u_2 = \lambda u_1\overline{u_2} 
	\end{aligned}
	\right.
\end{equation}
where $m_1,m_2>0$ and $\lambda \in \C$ with $m_1=2m_2$ and $|\lambda|=1$. 
Note that even if we consider a further generalized model by putting general nonzero complex constants in the front of both nonlinearities, it is reduced to the above model by change of variable (see \cite{HLN2011}).
One sees that the case $\lambda=1$ is the only case in which all solutions to the corresponding ODE remain bounded, and hence the Katayama–Sakoda framework applies.
Indeed, it will turn out that the corresponding ODE system
\begin{equation}\label{E:gODE}
	\left\{
	\begin{aligned}
	&i \partial_\tau A_1  = A_2^2,  \\
	&i \partial_\tau A_2  = \lambda A_1 \overline{A_2}
	\end{aligned}
	\right.
\end{equation}
admits a blowup solution when $\lambda\neq1$.
%

We state this result in more general form by 
introducing a generalized version. To make the symmetry of the system visible, we introduce it in the following form:
\begin{equation}\label{E:g3ODE}
	\left\{
	\begin{aligned}
	&i \partial_\tau A_1  = \lambda_1 \overline{A_2}\overline{A_3},  \\
	&i \partial_\tau A_2  = \lambda_2 \overline{A_3}\overline{A_1},\\
	&i \partial_\tau A_3  = \lambda_3 \overline{A_1} \overline{A_2},
	\end{aligned}
	\right.
\end{equation}
where $\lambda_j \in \C$ satisfies $|\lambda_j|=1$.
The system \eqref{E:gODE} above is the special case $\lambda_1=-1$ and $\lambda_2=\lambda_3=\lambda$ of \eqref{E:g3ODE}. Indeed, if $(A_1,A_2,A_3)$ is a solution to \eqref{E:g3ODE} under the choice of parameters then the pair $(\overline{A_1},A_2)$ is a solution to \eqref{E:gODE}.

We remark that only the relative arguments between three constants matter in \eqref{E:g3ODE}
since the change of variable $A_j \mapsto e^{i\theta_j} A_j$ ($j=1,2,3$) changes the constant $\lambda_j$ into $\lambda_j e^{i(\theta_1+\theta_2+\theta_3)}$ ($j=1,2,3$), namely, the three constants rotate in the same way. This also shows that the equation has the following gauge invariance property: If $(A_1,A_2,A_3)$ is a solution then, for any choice of $\theta_2,\theta_3 \in \R$,
$(e^{-i(\theta_2+\theta_3)}A_1, e^{i\theta_2}A_2, e^{i\theta_3}A_3)$
is also a solution (with the same coefficients).

\begin{theorem}\label{T:blowup}
If $|\lambda_1+ \lambda_2 + \lambda_3|>1$ then \eqref{E:g3ODE} admits a finite time blowup solution of the form
\[
	A_j (\tau)= z_j (1-\tau)^{-1+ i\psi_j} 
\]
with $\psi_1+\psi_2 + \psi_3=0$.
In particular, if two of $\lambda_j$ ($j=1,2,3$) are the same in addition, say if
$\lambda_1 = -1$ and $\lambda_2=\lambda_3 =e^{i\zeta}\neq 1$, then
the parameters can be given explicitly and the solution is given as follows:
\[
	A_1(\tau) = e^{i\zeta} ( \mu \sin \zeta-i)
	(1-\tau)^{-1- 2i  \mu \sin \zeta}, \quad
	A_2(\tau) = \sqrt{{3}{\mu} } (1-\tau)^{-1+ i \mu \sin \zeta},
\]
where
\[
	\mu = \tfrac2{ \sqrt{\cos^2 \zeta + 8} - 3\cos \zeta}>0.
\]
Under this specific choice of $\lambda_j$, the pair $(\overline{A_1}, A_2)$ is an exact blowing-up solution to \eqref{E:gODE} with $\lambda=\lambda_2$.
\end{theorem}

\begin{remark}
The assumption of the theorem is sharp in the following sense.
The case $-\lambda_1 = \lambda_2= \lambda_3$ of \eqref{E:g3ODE}, in which case $|\lambda_1+\lambda_2+\lambda_3|=1$, corresponds to \eqref{E:3ODE}.
Indeed, if the constants satisfy the condition, we may suppose that $-\lambda_1 = \lambda_2= \lambda_3=1$ due to the above change of variable. Then, the triplet $(\overline{A_1},A_2,A_3)$ solves \eqref{E:3ODE}.
As a result, as seen in Theorem \ref{T:sol_formula}, all solution is bounded and given explicitly.
\end{remark}

\begin{remark}
The assumption of the theorem can be rephrased as the existence of a line through the origin in the complex plane such that all $\lambda_1$, $\lambda_2$, $\lambda_3$ are located in the same half plane given by the line.
\end{remark}

Consider a generalized three-wave interaction model 
\begin{equation}\label{E:g3NLS}
	\left\{
	\begin{aligned}
	&i \partial_t u_1 + \frac1{2m_1}\partial_x^2 u_1 = \lambda_1 \overline{u_2} \overline{u_3}, \\
	&i \partial_t u_2 + \frac1{2m_2}\partial_x^2 u_2 = \lambda_2 \overline{u_3} \overline{u_1},\\
	&i \partial_t u_3 + \frac1{2m_3}\partial_x^2 u_3 = \lambda_3 \overline{u_1} \overline{u_2}.
	\end{aligned}
	\right.
\end{equation}
with $m_j\neq0$ and
$\lambda_j \in \C$ satisfying $m_1+m_2+m_3=0$ and
 $|\lambda|=1$.
The system \eqref{E:g3ODE} appears as
 the reduced ODE system corresponding to this NLS system.
 One consequence of Theorem \ref{T:blowup} is that if $|\lambda_1+\lambda_2+\lambda_3|>1$ then
the assumptions of the Katayama--Sakoda framework fail.
Namely, the asymptotic result similar to Theorem \ref{T:NLS} is not obtained, at least without any additional assumption.
To the best of the authors' knowledge, the asymptotic behavior of small solutions are not known for this generalized model in the case $|\lambda_1+\lambda_2+\lambda_3|\le 1$, except for 
\eqref{E:3ODE} (and those reduced this model) studied above.

\medskip

The rest of the paper is organized as follows.
In Section \ref{S:formula}, we prove Theorems \ref{T:sol_formula} and \ref{T:connection}.
We then prove Theorem \ref{T:blowup} in Section \ref{S:blowup}.
Finally, we treat the application to NLS (Theorem \ref{T:NLS})
in Section \ref{S:NLS}

\section{Explicit formula of solutions to \eqref{E:ODE} and \eqref{E:3ODE}}\label{S:formula}

In this section, we prove Theorems \ref{T:sol_formula} and \ref{T:connection}.
We first quickly prove Theorem \ref{T:connection}.
\begin{proof}[Proof of Theorem \ref{T:connection}]
Let $(A_1,A_2)$ be a solution to \eqref{E:ODE}.
Then, one observes that $(A_1,A_2,A_2)$ is a solution to \eqref{E:3ODE}.
Conversely, suppose that $(A_1,A_2,A_3)$ is a solution to \eqref{E:3ODE} and $A_2=A_3$ for all time. Then, substituting the identity to \eqref{E:3ODE}, we see that $(A_1,A_2)$ solves \eqref{E:ODE}.
\end{proof}

In the rest of this section, we prove Theorem \ref{T:sol_formula}.

\subsection{Preliminaries}

We 
recall a well-known integral (cf. \cite[233.00]{BFBook}).
For self-containedness, we give a proof.
\begin{lemma}\label{L:EllipticIntegral}
    For $a<b<c$, one has
\begin{equation}\label{E:formula1}
    \int_a^x \frac{dy}{\sqrt{(y-a)(b-y)(c-y)}} = \frac{2}{\sqrt{c-a}} F\left( \arcsin\sqrt\frac{x-a}{b-a} ; \sqrt{\frac{b-a}{c-a}}\right)
\end{equation}
for $x\in [a,b]$,
where $F(\phi;k)=\int_0^\phi \frac{d\theta}{\sqrt{1-k^2 \sin^2 \theta}}$ is the incomplete elliptic integral of first kind.
Further, if we denote the right-hand side of \eqref{E:formula1} as $u$, we have
\begin{equation}\label{E:formula2}
    x=x(u)= a+ (b-a) \sn^2 ( \tfrac{\sqrt{c-a}}{2}u;\sqrt{\tfrac{b-a}{c-a}}).
\end{equation}
\end{lemma}
\begin{proof}
The identity \eqref{E:formula1} follows from by the change of variable $y = a + (b-a) \sin^2 \phi$.

If the right-hand side of \eqref{E:formula1} is denoted as $u\ge0$, \eqref{E:formula1} is written as
\[
    \sqrt{\tfrac{x-a}{b-a}} = \sn ( \tfrac{\sqrt{c-a}}{2}u;\sqrt{\tfrac{b-a}{c-a}})
\]
by definition of Jacobi elliptic function, from which we obtain \eqref{E:formula2}.
 \end{proof}

\subsection{Proof for the case $\gamma=0$}

Let us first consider the case $\gamma=0$. 
This case is handled in \cite{OU2015,Ur16}. 

It is easy to see that $(A_1,0,0)$ is an equilibrium point for any $A_1\in \C$.
Suppose that $(A_1,A_2,A_3)$ is not an equilibrium, i.e., suppose that $A_2(0)>0 $.

Let $(B_1,B_2,B_3)$ be a real-valued solution to
\begin{equation}\label{E:Beq}
	B_1' = B_2B_3, \quad B_2' = - B_1B_3, \quad B_3' = - B_1B_2.
\end{equation}
One sees that $(-iB_1,B_2,B_3)$ is a solution to \eqref{E:3ODE} with $\gamma=0$.
Thanks to the well-known formula
\begin{align*}
	\tfrac{d}{d\tau} \sn \tau &= \cn \tau \dn \tau, &
	\tfrac{d}{d\tau} \cn \tau &= -\sn \tau \dn \tau,&
	\tfrac{d}{d\tau} \dn \tau &= -k^2\sn \tau \cn \tau
\end{align*}
(e.g. \cite[731.01,731.02,731.03]{BFBook}),
we see that, for any $\lambda>0$, $k\in [0,1]$, and $\tau_0\in \R$,
the triplet
\[
	B_1(\tau) = \lambda k \sn (\lambda (\tau-\tau_0) ; k),\quad
	B_2(\tau) = \lambda \dn (\lambda (\tau-\tau_0) ; k),\quad
	B_3(\tau) = \lambda k \cn (\lambda (\tau-\tau_0) ; k)
\]
solves \eqref{E:Beq}.
Suppose that $(B_1,B_2,B_3)$ satisfies \eqref{E:inv}. Then,
\[
	1= |B_1(\tau_0)|^2 + \tfrac12 |B_2(\tau_0)|^2 + \tfrac12|B_3(\tau_0)|^2= \tfrac12 \lambda^2 (k^2 +1).
\]
Under this constraint, we have $c= \lambda^2 k^2$. Hence, we have
$k= \sqrt{\frac{c}{2-c}}$ and $\lambda = \sqrt{2-c}.$
One sees that we can choose $\tau_0$ so that $(-i B_1(0),B_2(0),B_3(0)) = (A_1(0),A_2(0),A_3(0))$.

%
\subsection{Proof for the case $\gamma\neq0$}

We proceed to the main case $\gamma\neq0$.
Although it is sufficient to check that the given formula actually is a solution, we here give a constructive proof based on the argument in \cite{Ma2}.

\begin{proof}[Proof of Theorem \ref{T:sol_formula} under $\gamma \neq0$]
Pick a solution $(A_1,A_2,A_3)$ to \eqref{E:3ODE} satisfying \eqref{E:inv} and $\gamma \neq0$.

\subsubsection*{Step 1}
We first specify the behavior of quadratic quantities.
One sees from \eqref{E:3ODE} that
\[
    \rho_1' = - 2 \Im (A_1\overline{A_2}\overline{A_3}).
\]
Hence,
\[
	(\rho_1' )^2 = 4 (\Im (A_1 \overline{A_2} \overline{A_3}))^2 = 4(\rho_1\rho_2\rho_3 -{\gamma}^2).
\]
Thus, 
using the fact that $\rho_1+\rho_3=c$ and $\rho_1+\rho_2=2-c$ are conserved,
we have the following relation
\begin{equation}\label{E:rho1ODE1}
   0\le (\rho_1')^2 
    =4 f(\rho_1).
\end{equation}

In the extremal case $\gamma^2=m>0$,
$f(\rho_1)\le0 $ holds for all $\rho_1 \in[0, c]$.
Hence, we see from \eqref{E:rho1ODE1} that
$f(\rho_1)=0$
for all time.
Hence, $\rho_1(\tau)=\rho_1(0)=e_1 $ for all time.

Let us consider the other case $0<\gamma^2<m$.
In this case, we have $$0<e_1<e_2<c\le 2-c <e_3<2.$$
Since the right hand side of \eqref{E:rho1ODE1} is nonnegative, we have
$\rho_1 \in [ e_1, e_2]$ for all time.
This implies that the orbit of $(\rho_1,\rho_1')$ is a subset of a closed curve 
\[
	\{ (x,y)\in [e_1,e_2]\times \R ; y^2 = 4f(x)\}.
\]
Further, since
\[
	\rho_1'' = (-2\Im (A_1 \overline{A_2} \overline{A_3}))' = 2 (3\rho_1^2 -4\rho_1+c(2-c))=2f'(\rho_1),
\]
we see that $\rho_1'(\tau)=0$ implies that $\rho_1(\tau) \in \{ e_1, e_2\}$, and hence that
$\rho_1''(\tau)\neq0$.
Hence the vector field does not vanish on the level curve.
Since the system is autonomous and the orbit is a compact level set without equilibrium points, the solution is periodic.

Let us find the explicit formula.
Since the solution is periodic,
it suffices to derive the formula on a single period.
Pick a time $\tau_0$ so that $\rho_1(\tau_0)=e_1$.
For simplicity, we may suppose that $\tau_0=0$.
Then, we have $\rho_1(0)= e_1$.
We have $\rho_1'(0)=0$ and $\rho_1''(0)=2f'(e_1)>0$.
Then, \eqref{E:rho1ODE1} is written as
\[
    \rho_1' =\begin{cases} 2 \sqrt{ (\rho_1-e_1)(e_2-\rho_1)(e_3-\rho_1)} & (\tau>0),\\
 -2 \sqrt{(\rho_1-e_1)(e_2-\rho_1)(e_3-\rho_1)} & (\tau<0)
 	\end{cases}
\]
for small $\tau$. This implies
\[
    \int_{e_1}^{\rho_1(\tau)} \frac{d x}{\sqrt{ (x-e_1)(e_2-x)(e_3-x)}}
    = 2 |\tau|
\]
for sufficiently small $|\tau|$.
Thus, by Lemma \ref{L:EllipticIntegral}, we see that
\begin{equation}\label{E:Formula pf1}
\begin{aligned}
    \rho_1(\tau) &= e_1 + (e_2-e_1) \sn^2
    \left( \sqrt{e_3-e_1}\tau,\sqrt{\tfrac{e_3-e_2}{e_3-e_1}}\right)\\
   & = e_1 (1-n_1 \sn^2 ( \omega \tau ; k) )
   \end{aligned}
\end{equation}
for small $\tau$.
Note that this expression is valid as long as $\rho_1$ is monotone.
Hence, it is valid on $[-\frac{1}{\omega}K(k),\frac{1}{\omega}K(k)]$, where $K(k)$ is the complete elliptic integral.
We have
\[
	\rho_1(-\tfrac{1}{\omega}K(k))
	= e_2 = \rho_1(\tfrac{1}{\omega}K(k))\quad \text{and} \quad
	\rho_1'(-\tfrac{1}{\omega}K(k))
	=0= \rho_1'(\tfrac{1}{\omega}K(k)).
\]
Hence, by the uniqueness of the solution to ODE,
one verifies that the period of $\rho_1$ is 
$\tfrac{2}{\omega}K(k)$.
Since $\sn^2(\omega \tau;k)$ has the same period 
$\tfrac{2}{\omega}K(k)$, and since the expression 
\eqref{E:Formula pf1} agrees with $\rho_1$ on 
$[-\tfrac{1}{\omega}K(k), \tfrac{1}{\omega}K(k)]$, 
we conclude that \eqref{E:Formula pf1} is valid on $\R$.
\subsubsection*{Step 2}

Once we have $\rho_1$, we can reconstruct $A_1$, $A_2$, and $A_3$.
We treat the whole $0<\gamma^2\le m$ case in a unified way.
Note that the expression \eqref{E:Formula pf1} is valid even when $\gamma^2= m$, in which case $e_1=e_2$ and hence $n_1=0$.

Let $\Phi_j (t)$ be the phase of $A_j(t)$. Then, considering the imaginary part of the identity
\[
    \overline{A_1} \cdot iA_1' = i(\sqrt{\rho_1})' \sqrt{\rho_1} - \Phi_1' \rho_1,
\]
one has
\[
    -\Phi_1' \rho_1 = \Re (\overline{A_1}A_2A_3) =\gamma,
\]
namely
\[
    \Phi_1(\tau) = \Phi_1(0) - \int_0^\tau \frac{\gamma}{\rho_1(s)}\,ds.
\]
Hence, we have the solution formula
\[
    A_1(\tau) = \sqrt{\rho_1(\tau)}\exp \left(i\Phi_1(0)  - i\int_0^\tau \frac{\gamma}{\rho_1(s)}\,ds\right).
\]
Recalling that $A_2(0),A_3(0) \in \R_{\ge0}$, we see that
\begin{align*}
	\gamma = \Re (\overline{A_1}(0)A_2(0)A_3(0))
	&= \cos \Phi_1(0) \sqrt{\rho_1(0)(2-c-\rho_1(0))(c-\rho_1(0))} \\
	&= \cos \Phi_1(0) \sqrt{f(\rho_1(0)) + \gamma^2}.
\end{align*}
Thus,
\[
	\Phi_1(0) = \arccos \tfrac{\gamma}{\sqrt{f(\rho_1(0))+\gamma^2}}.
\]
Then, it suffices to prove that $\Phi_1(\tau)$ is a primitive of $-\gamma/\rho_1$.
However, it is readily seen
by \eqref{E:Pirelation}.

Let us consider the representation of $A_2$. 
The formula of
$\rho_2$ is obtained by \eqref{E:Formula pf1} and $\rho_1+\rho_2=2-c$.
Let us consider the phase part. As above, we have
\[
    -\Phi_2' \rho_2 = \Re (A_1  \overline{A_2} \overline{A_3}) = \gamma.
\]
and so
\[
    \Phi_2(\tau) = - \int_0^\tau \frac{\gamma}{2-c-\rho_1(s)}\,ds.
\]
Hence, we obtain
\[
    A_2(t) = \sqrt{2-c-\rho_1(\tau)} \exp \left(  - i\int_0^\tau \frac{\gamma}{2-c-\rho_1(s)}\,ds\right).
\]
One deduces from \eqref{E:Pirelation} that $\Phi_2(\tau)$ is a primitive of $-\gamma/(2-c-\rho_1)$.
Similarly, we deduce from $\rho_1+\rho_3=c$ that
\[
    A_3(t) = \sqrt{c-\rho_1(\tau)} \exp \left(  - i\int_0^\tau \frac{\gamma}{c-\rho_1(s)}\,ds\right).
\]
It follows that $\Phi_3(\tau)$ is a primitive of $-\gamma/(c-\rho_1)$.
This completes the proof.
\end{proof}

\section{Blowup solutions to \eqref{E:gODE} and \eqref{E:g3ODE}}\label{S:blowup}

In this section, we prove the existence of a blowup solution to \eqref{E:gODE} and \eqref{E:g3ODE}.
\subsection{An equivalent characterization of the conditions}
Let us first introduce an alternative characterization of the assumption of the theorem.

\begin{lemma}\label{L:blowup}
Suppose that $|\lambda_j|=1$ ($j=1,2,3$).
The condition $|\lambda_1+\lambda_2+\lambda_3|>1$ holds if and only if
there exist a suitable relabeling of the indices and a rotation $\lambda_j\mapsto e^{i\zeta} \lambda_j $ ($j=1,2,3$) such that $\{\lambda_j\}$ is given as
$\lambda_j = e^{i\zeta_j}$ with $0= \zeta_1 \le \zeta_2 \le \zeta_3 <\pi$.
\end{lemma}
\begin{proof}
Pick $\{\lambda_j\}_{j=1}^3$ so that $|\lambda_j|=1$ for each $j$.
By rotational symmetry, such a triple is essentially determined by a decomposition of $2\pi$ into three nonnegative angles. Let $\theta$ denote the largest of these angles. Then $\theta \in [\tfrac{2\pi}{3}, 2\pi]$.

Using this angle $\theta$, and after a suitable relabeling together with a rotation if necessary, we may write
\[
	\lambda_1 = e^{i\zeta}, \quad
	\lambda_2 = e^{i(\pi-\frac\theta2)}, \quad
	\lambda_3 = e^{i(\pi+\frac\theta2)}.
\]
Here, $\zeta$ satisfies
\[
	\max (\pi -\tfrac32 \theta,-\pi + \tfrac12 \theta) \le \zeta \le \min (-\pi + \tfrac32 \theta , \pi - \tfrac12 \theta).
\]

Let us claim that $|\lambda_1+\lambda_2+\lambda_3|>1$ holds if and only if $\theta>\pi$.
If $\theta=\pi$ then $|\lambda_1+\lambda_2+\lambda_3|=|\lambda_1|=1$.
If $\theta \in (\pi,2\pi]$ then $\lambda_2+\lambda_3= -2 \cos \frac{\theta}2>0$ and $\cos \zeta>0$ and hence
\[
	|\lambda_1+\lambda_2+\lambda_3|^2
	= (\cos \zeta-2 \cos \tfrac{\theta}2)^2 + \sin^2 \zeta > \cos^2 \zeta + \sin^2 \zeta =1.
\]
If $\theta \in [\frac23\pi,\pi)$ then $|\zeta|
\le -\pi +\frac32 \theta< \frac12\theta<\frac\pi2$ and hence $\cos \zeta > \cos \frac{\theta}2>0$. Therefore,
\[
	|\lambda_1+\lambda_2+\lambda_3|^2
	= 1-4 \cos \tfrac{\theta}2(  \cos \zeta - \cos \tfrac{\theta}2)<1.
\]
Hence, the claim is established.

In the case $\theta>\pi$, we relabel the indices so that the  $(\lambda_3,\lambda_1,\lambda_2)$ is new $(\lambda_1,\lambda_2,\lambda_3)$. Then, rotating them to obtain $\lambda_1=1$, $\lambda_2=e^{i(\zeta+\pi -\frac\theta2)}$, and $\lambda_3 = e^{i(2\pi-\theta)}$, i.e., $\zeta_1=0$, $\zeta_2=\zeta+\pi -\frac\theta2$, and $\zeta_3=2\pi-\theta$. Then, since $|\zeta|\le \pi -\frac12 \theta$ and $\theta>\pi$, one has $0 = \zeta_1 \le \zeta_2 \le \zeta_3 < \pi$.
\end{proof}

\subsection{Proof of Theorem \ref{T:blowup}}
Let us turn to the Proof of Theorem \ref{T:blowup}. The proof is divided into two parts. In the first part, we consider the case where two of $\lambda_j$ are the same.
\begin{proof}[Proof of Theorem \ref{T:blowup} -- Part 1]
Suppose that two of $\lambda_j$ are the same.
By relabeling and using rotation symmetry if necessary, we may suppose that
$\lambda_1= -1$ and $\lambda_2=\lambda_3=e^{i\zeta}\neq 1$.
Note that $e^{i\zeta}= 1$ is excluded since $|\lambda_1+\lambda_2+\lambda_3|=1$ in this case.

We use the ansatz
\[
    A_1 (\tau) = z_1 (1-\tau)^{-1-2i\psi}, \quad
    A_2 (\tau) = A_3(\tau)= z_2 (1-\tau)^{-1+i\psi}
\]
where $z_1 \in \C$, $z_2>0$, and $\psi \in \R$ are constants.

We shall choose the constants suitably.
The ODE system \eqref{E:g3ODE} reads as
\begin{equation}\label{E:BUpf1}
    (i+2\psi)z_1 = z_2^2, \quad
    i+\psi  = e^{i\zeta} z_1.
\end{equation}
Let us find suitable constants satisfy this relation.
By combining these two, we have
\begin{equation}\label{E:BUpf2}
    e^{i\zeta}z_2^2 =(i+2\psi)(i+\psi)
    =(2\psi^2-1) + 3i\psi.
\end{equation}
When $e^{i\zeta}=-1$, we see that $(z_1,z_2,\psi)=(-i,1,0)$ is a solution to \eqref{E:BUpf1}.
Suppose that $\sin \zeta \neq 0$.
By the comparison of the real and the imaginary parts, we see that
\[
    \tfrac{2\psi^2-1}{3\psi} = \cot \zeta.
\]
This is reduced to a quadratic equation with respect to $\psi$.
Noting that the sign of the $\psi$ is the same as $\sin \zeta$ in view of \eqref{E:BUpf2}, we choose a suitable branch to obtain
\[
    \psi = \tfrac{3\cot \zeta + \sqrt{9\cot^2 \zeta + 8}}{4}
    = \tfrac{3\cos \zeta + \sqrt{\cos^2 \zeta + 8}}{4\sin \zeta}
    = \mu \sin \zeta.
\]
where $\mu = \tfrac2{ \sqrt{\cos^2 \zeta + 8} - 3\cos \zeta} >0$.
Then, it is straightforward to see that
\[
    z_1 = e^{-i\zeta} (\mu \sin \zeta+i)
\]
and
\[
    z_2 = |2\psi^2-1 + 3i\psi|^{\frac12} =\sqrt{3|\psi|} |\cot \zeta+i|^{\frac12}
    =  \sqrt{3{\mu}}
\]
follow from \eqref{E:BUpf1} and \eqref{E:BUpf2}, respectively.
Hence, we obtain the desired formula.
\end{proof}
Let us consider the other case.
\begin{proof}[Proof of Theorem \ref{T:blowup} -- Part 2]
Suppose that $\lambda_j$ are mutually different.
By Lemma \ref{L:blowup}, we may suppose that $\lambda_j = e^{i\zeta_j}$ with $0=\zeta_1 < \zeta_2 < \zeta_3 <\pi$ without loss of generality.

We use the ansatz
\[
    A_1 (\tau) = z_1 (1-\tau)^{-1-i(\psi_2+\psi_3)}, \quad
    A_2 (\tau) = z_2 (1-\tau)^{-1+i\psi_2}, \quad
    A_3 (\tau) = z_3 (1-\tau)^{-1+i\psi_3},
\]
where $z_1 \in \C\setminus\{0\}$, $z_2>0$, $z_3>0$, and $\psi_j \in \R$ are constants.
The ODE system \eqref{E:g3ODE} reads as
\begin{equation}\label{E:gBUpf1}
    (i-(\psi_2+\psi_3)) z_1  = z_2z_3, \quad
    (i+\psi_2)z_2 = \lambda_2 \overline{z_1}z_3, \quad
    (i+\psi_3)z_3 = \lambda_3 \overline{z_1}z_2.
\end{equation}
As in the first part, we shall choose the constants suitably.
From the first and the third identities, one has
\begin{equation}\label{E:gBUpf2}
    \overline{\lambda_3}z_2^2
     = 1-\psi_3(\psi_2+\psi_3)  + i (\psi_2+2\psi_3).
\end{equation}
Similarly, we also have
\begin{equation}\label{E:gBUpf3}
    \overline{\lambda_2}z_3^2      = 1 - \psi_2(\psi_2+\psi_3) + i (2\psi_2 + \psi_3).
\end{equation}
%
%
%
By the comparison of the real and the imaginary parts of \eqref{E:gBUpf2}, we see that
\[
    \tfrac{\psi_3(\psi_2+\psi_3)-1}{\psi_2 + 2\psi_3} = \cot \zeta_2.
\]
This is written as
\begin{equation}\label{E:gBUpf4}
	\psi_2 = -(\psi_3- \cot \zeta_2) + \tfrac{1+ \cot^2 \zeta_2}{\psi_3-\cot \zeta_2}
\end{equation}
Similarly, \eqref{E:gBUpf3} yields
\begin{equation}\label{E:gBUpf5}
	\psi_3 = -(\psi_2- \cot \zeta_3) + \tfrac{1+ \cot^2 \zeta_3}{\psi_2-\cot \zeta_3}
\end{equation}
We look for a solution $(\psi_2,\psi_3)$ to \eqref{E:gBUpf4}--\eqref{E:gBUpf5} such that 
$\psi_2+2\psi_3<0$ and $2\psi_2+\psi_3<0$.
The last two conditions come from the comparison of the imaginary parts of \eqref{E:gBUpf1} and \eqref{E:gBUpf2}.

To find such a pair $(\psi_2,\psi_3)$, we first note that
the curves
\[
	C_1:= \{ 	y = -(x- \cot \zeta_2) + \tfrac{1+ \cot^2 \zeta_2}{x-\cot \zeta_2},\, x< \cos \zeta_2
\}
\]
and
\[
	C_2:=\{ 	x = -(y- \cot \zeta_3) + \tfrac{1+ \cot^2 \zeta_3}{y-\cot \zeta_3},\, y< \cos \zeta_3
\}
\]
has a unique intersection point.
This can be seen, for instance, from the fact that
\[
	(-\infty , \cot \zeta_2)\ni x \mapsto y_1(x):=-(x- \cot \zeta_2) + \tfrac{1+ \cot^2 \zeta_2}{x-\cot \zeta_2}
\]
is strictly monotone decreasing and
\[
	\lim_{x\to-\infty}
	y_1(x)=\infty,\quad
	\lim_{x\to\cot \zeta_2-0}
	y_1(x)=-\infty
\]
and that $C_2$ is written as a graph of 
a strictly monotone decreasing function $\R \ni x \mapsto y_2(x)$ such that
$\lim_{x\to-\infty}y_2(x) = \cot \zeta_3$.
Further, $C_1$ is below the line $y=-2x$ in the $xy$-plane.
Indeed, if there exists an intersection point then its $x$-coordinate is a solution to $x=y_1(x)$.
Since the equation is written as
$x^2+1 = 0$, the existence yields a contradiction.
Together with $y_1(x) \to -\infty$ as $x\to \cot \zeta_2-0$, we see that $C_1$ is below the line. 
Similarly, we see that $C_2$ is below the line $y=-\frac12 x$.
Thus, the unique intersection of $C_1$ and $C_2$ is below the two lines. This is our desired pair $(\psi_2,\psi_3)$.
\end{proof}

%

\section{Proof of Theorem \ref{T:NLS}}\label{S:NLS}
In this section, we sketch the proof of Theorem 1.3,
emphasizing the parts specific to the present problem.
The argument is based on the framework developed in \cite{Ma3,MaSeUr3} (see also \cite{MaSeUr2}).
To this end, we make notations. 
For $m\in \R\setminus\{0\}$, we introduce the scaled Fourier transform $\mathcal{F}_m$ by
\begin{equation*}
(\mathcal{F}_m\phi)(\xi):=\frac{m}{2\pi}\int_{\R^2}e^{-imy\cdot\xi}\phi(y)\,dy,
\quad \xi\in \R^2.
\end{equation*}
One can easily check that $\mathcal{F}_m$ is an isometry on $L^2(\R^2)$.
We denote $\F=\F_1$, for short.
One has $\F_{m}^{-1} = - \F_{-m}$.
For non-negative integers $s$ and $\sigma$, we define weighted Sobolev space by
$$
H^{s,\sigma}(\R^2; \C^3)=\{f\in L^2(\R^2; \C^3); \|f\|_{H^{s,\sigma}}<\infty\},
$$
where 
$$
\|f\|_{H^{s,\sigma}}^2=\sum_{|\alpha|\le s} \int_{\R^2} \Jbr{x}^{2\sigma} |\pa_x^\alpha f(x)|^2\, dx
$$
with $\Jbr{x}=\sqrt{1+|x|^2}$.

For $m\neq0$,
we let $U_m(t) = e^{i\frac{t}{2m}\Delta}= \F^{-1} e^{-i\frac{t}{2m}|\xi|^2}\F 
$ and
$J_m(t) = U_m(t) x U_m(-t)$.
Recall that
$$\epsilon := \sum_{j=1}^3\|u_{j,0}\|_{H^{0,2}}. $$

\subsection{Local existence}
Let us start with local existence. 
By the standard theory of local well-posedness for
the nonlinear Schr\"{o}dinger equation
(see \cite{Caz} for instance),
we have the following.

\begin{proposition}\label{prop:local}
    There exists $\widetilde{\epsilon}_0 > 0$ such that if
    $\epsilon \le \widetilde{\epsilon}_0$, then there exists
    a unique solution
    $(u_1, u_2,u_3) \in (C([-1, 1];L^2(\R^2)))^3$ to \eqref{E:3NLS} such that
    \[
         U_{m_j}(-t)u_j \in C([-1,1];H^{0,2}(\R))
    \]
    and
    \begin{equation}
        \max_{t \in [-1,1]}\sum_{j = 1}^3 
    \|U_{m_j}(-t)u_j(t)\|_{H^{0,2}}
     \le 2\epsilon.
    \end{equation}
\end{proposition}

\subsection{Global existence}

In what follows, we only consider positive time.
We shall next show the existence of global solution by
proving that $\sum_{j=1}^3\|U_{m_j}(-t)u_j\|_{H^{0,2}}$ remains uniformly bounded for all all $t>0$.

For $\nu > 0$ and $T > 1$, we define
\begin{equation}
    X_T := \sup_{t \in [1,T]}t^{-\nu\epsilon^2}
    \sum_{j = 1}^3
    \|U_{m_j}(-t)u_j(t)\|_{H^{0,2}}
\end{equation}
and
\begin{equation}
    Y_T := \sup_{t \in [1,T]}t\sum_{j=1}^3\|u_j(t)\|_{L^\infty}.
\end{equation}
The following proposition yields global existence.

\begin{proposition}\label{prop:global}
    Let $\widetilde{\epsilon_0}$ be given by Proposition
    \ref{prop:local}.
    Then there exist $\nu > 0$, $C > 0$, and $\epsilon_0 \in (0,\widetilde{\epsilon_0}]$
    such that if $\epsilon\le\epsilon_0$, then the unique solution to \eqref{E:3NLS}
    given in Proposition \ref{prop:local} exists globally
    in time for positive time direction and obeys the bound
    \begin{equation}\label{E:global}
        \sup_{T \ge 1}\left(X_T + Y_T\right) \le C\epsilon.
    \end{equation}
\end{proposition}
It can be proved by showing that there exist $\nu > 0$, $C_0 > 0$, $\widetilde{C}_0 > 0$, and $\epsilon_0 > 0$ such that if
                $\epsilon \in (0,\epsilon_0]$ and
    \begin{equation}\label{eq:boot1}
        X_T \le C_0\epsilon, \quad Y_T \le \widetilde{C}_0\epsilon
    \end{equation}
    for some $T \ge 1$ then it holds that
    \begin{equation}\label{eq:boot2}
        X_T \le \tfrac{1}{2}C_0\epsilon, \quad Y_T \le \tfrac{1}{2}\widetilde{C}_0\epsilon
    \end{equation}
	for the same $T$.
	
	One can obtain the estimate on $X_T$ in \eqref{eq:boot2} by a standard energy-type estimate. The sharp $L^\infty$-decay estimate, obtained by the boundedness of $Y_T$, allows us to close the estimate without any loss of time growth rate. Note that refinement of the coefficient in \eqref{eq:boot2} is achieved by taking $\nu$ large compared with $C_0$ and $\widetilde{C}_0$.
	
	To obtain the estimate for $Y_T$, we introduce a new variable $$w_j:= \F_{m_j} U_{m_j}(-t)u_j.$$
	One sees that $t\|u_j(t)\|_{L^\infty}$ is comparable to $\|w_j(t)\|_{L^\infty}$ up to error $O(\eps t^{-\frac12+\nu \eps^2} X_T)$ (see \eqref{E:asymp_pre}, below).
	Further, thanks to
	the resonant condition $m_1=m_2+m_3$, we obtain
	\[
		i \partial_t w_j = \frac1{t} F_j(w_1,w_2,w_3) + r_j
	\]
	for $j=1,2,3$ in the sense of the associated integral equation, where $F_j$ is the nonlinearity of \eqref{E:3NLS} and $r_j$ is the error term which is order $O_{L^\infty} (\eps^2 t^{-\frac32 + \nu \varepsilon^2}X_T^2)$ on $t\in [1,T]$. 
After introducing the logarithmic time variable ($\tau=\log t$),
the leading part of the equation reduces to the ODE system \eqref{E:3ODE}.	
Hence, the conservation of $\rho$ for \eqref{E:3ODE} 
here reads as
\[
		\partial_t (|w_1|^2 + \tfrac12 |w_2|^2 + \tfrac12|w_3|^2) = O(|w||r| ).
\]
This is the crucial part of the argument.
By integration, we obtain the estimate on $Y_T$. The refinement of coefficient in \eqref{eq:boot2} can be obtained by choosing $\widetilde{C}_0$ large compared with $\varepsilon^{-1} Y_1$. 

\subsection{Asymptotic behavior}

The global bounds obtained above allow us
to identify the asymptotic profile.
We further introduce $\alpha = (\alpha_1,\alpha_2,\alpha_3)$ by
\begin{equation*}
    \alpha_j(\log t,\xi) := w_j(t,\xi).
\end{equation*}
This satisfies
	\begin{equation}\label{E:alpha_eq}
		i \partial_t \alpha_j =  F_j(\alpha_1,\alpha_2,\alpha_3) + \tilde{r}_j
	\end{equation}
	(in the integral form), where $\tilde{r}_j$ is the error term given by $\tilde{r}_j(\log t) = r_j(t)$.
With the global bound \eqref{E:global} in hand, we obtain the following estimates.
\begin{proposition}
	If $\eps \le \eps_0$ then
  it holds that 
  \begin{equation}\label{E:asymp_pre}
    \left\| u_j(t) - (it)^{-1}e^{im_j \frac{|x|^2}{2t}}\alpha_j\left(\log t, \frac{\cdot}{t}\right)\right\|_{L^p} \lesssim \epsilon t^{-\frac32 + \frac1p +\nu\epsilon^2}
  \end{equation}
  for $p \in [2,\infty]$, $t\ge1$ and
  \begin{equation}\label{E:error}
  	\| \tilde{r}(\tau) \|_{L^p} \le C_1 \eps^2 e^{-(\frac12 + \frac1p - \nu \eps^2)\tau}
  \end{equation}
  for $p \in [2,\infty]$, $\tau \ge 0$.
\end{proposition}

Now, our final task is to construct a family $A=(A_1,A_2,A_3)$ of solutions to \eqref{E:3ODE} such that $|A(\tau)-\alpha(\tau)| \to 0$ as $\tau\to\infty$.

If the nonlinearity is removed by the gauge transformation, the construction is simple.
However, this is not the case.
Hence, we use a different argument.
We let
\[  
	||| A||| := \| e^{(1-\nu \eps^2)\tau} A \|_{L^\infty_\tau ([0,\infty); L^2(\R^2_\xi))}
	+ \| e^{(\frac12-\nu \eps^2)\tau} A \|_{L^\infty_{\tau,\xi} ([0,\infty)\times \R^2))}
\]
and introduce a complete metric space 
\[
	Z:=\left\{ A=(A_1,A_2,A_3) \in (C^\infty_\tau([1,\infty); L^2_\xi \cap L^\infty_\xi(\R^2)))^3 \ ;\ ||| A-\alpha||| \le C_2 \eps^2 \right\}
\]
with metric $d(A,B) = |||A-B|||$.
Note that $\alpha \in Z$ and hence $Z$ is non-empty.
We see that
\[
	Z\ni A \mapsto \Phi(A)(\tau) :=\alpha(\tau) -i\int_\tau^\infty \tilde{r}(\sigma)d\sigma +i \int_\tau^\infty (F (A)(\sigma)-F(\alpha) (\sigma)) d\sigma
\]
is a contraction map on $Z$.
Indeed, 
using the bound $\|\alpha\|_{L^\infty_{\tau,\xi}} \le Y_\infty \le \frac12 \widetilde{C}_0 \eps$ and \eqref{E:error}, we obtain
\[
	||| \Phi(A) - \alpha ||| \lesssim C_1 \eps^2 + (\tilde{C}_0\eps + C_2 \eps^2
 ) C_2 \eps^2
\]
for any $A \in Z$ and
\[
	d (\Phi(A_1) , \Phi(A_2) )
	\lesssim (\tilde{C}_0\eps )d(A_1,A_2).
\]
Hence, by taking $C_2$ large compared with $C_1$ and letting $\eps_0$ small, we see that $\Phi$ is a contraction map on $Z$.

Let $A\in Z$ be the fixed point of $\Phi$.
Then, in view of \eqref{E:alpha_eq}, we have
\[
	A(\tau_2)-A(\tau_1) = \left(\alpha(\tau_2)-\alpha(\tau_1) + i \int^{\tau_2}_{\tau_1} (F(\alpha)+\tilde{r})d\sigma \right) -i \int^{\tau_2}_{\tau_1} F(A)
	d\sigma
	= -i \int^{\tau_2}_{\tau_1} F(A)
	d\sigma
\]
for any $ \tau_1, \tau_2 \ge 0 $. This shows that $A$ is a solution to \eqref{E:3ODE}.
The property $|||A-\alpha||| \le C_2 \eps^2$ reads as
\[
	\left\|  (it)^{-1}e^{im_j \frac{|x|^2}{2t}}\alpha_j\left(\log t, \frac{\cdot}{t}\right) - (it)^{-1}e^{im_j \frac{|x|^2}{2t}}A_j\left(\log t, \frac{\cdot}{t}\right)\right\|_{L^p} \le C_2 \eps^2 t^{-\frac32 + \frac1p +\nu\epsilon^2}.
\]
Hence, together with \eqref{E:asymp_pre}, we obtain \eqref{E:asymp}.
Theorem \ref{T:sol_formula} now provides an explicit description
of the reduced dynamics appearing in the asymptotic profile.

\subsection*{Acknowledgements} 
S.M. was  supported by JSPS KAKENHI Grant Numbers JP23K20803, JP23K20805, and JP24K00529.
K.U. was  supported by JSPS KAKENHI Grant Numbers JP24K00529 and JP24K06805.


\begin{bibdiv}
\begin{biblist}

\bib{ABDP}{article}{
   author={Armstrong, J.~A.},
   author={Bloembergen, N.},
   author={Ducuing, J.},
   author={Pershan, P.~S.},
   title={Interactions between light waves in a nonlinear dielectric},
   journal={Phys. Rev.},
   volume={127},
   date={1962},
   number={6},
   pages={1918--1939},
   doi={10.1103/PhysRev.127.1918},
}


\bib{BFBook}{book}{
   author={Byrd, Paul F.},
   author={Friedman, Morris D.},
   title={Handbook of elliptic integrals for engineers and scientists},
   series={Die Grundlehren der mathematischen Wissenschaften, Band 67},
   note={Second edition, revised},
   publisher={Springer-Verlag, New York-Heidelberg},
   date={1971},
   pages={xvi+358},
   review={\MR{0277773}},
}

\bib{Caz}{book}{
   author={Cazenave, Thierry},
   title={Semilinear Schr\"{o}dinger equations},
   series={Courant Lecture Notes in Mathematics},
   volume={10},
   publisher={New York University, Courant Institute of Mathematical
   Sciences, New York; American Mathematical Society, Providence, RI},
   date={2003},
   pages={xiv+323},
   isbn={0-8218-3399-5},
   review={\MR{2002047}},
}
%
%

\bib{GMXZ}{article}{
   author={Gao, Chuanwei},
   author={Meng, Fanfei},
   author={Xu, Chengbin},
   author={Zheng, Jiqiang},
   title={Scattering theory for quadratic nonlinear Schr\"odinger system in
   dimension six},
   journal={J. Math. Anal. Appl.},
   volume={541},
   date={2025},
   number={1},
   pages={Paper No. 128708, 42},
   issn={0022-247X},
   review={\MR{4781063}},
   doi={10.1016/j.jmaa.2024.128708},
}

\bib{GiOz}{article}{
   author={Ginibre, J.},
   author={Ozawa, T.},
   title={Long range scattering for nonlinear Schr\"{o}dinger and Hartree
   equations in space dimension $n\geq 2$},
   journal={Comm. Math. Phys.},
   volume={151},
   date={1993},
   number={3},
   pages={619--645},
   issn={0010-3616},
   review={\MR{1207269}},
}

\bib{Ham2018}{misc}{
   author={Hamano, Masaru},
	title={Global dynamics below the ground state for the quadratic Schr\"odinger system in 5$d$},
	year={2018},
	 status={available as arXiv:1805.12245}
}

\bib{HIN}{article}{
   author={Hamano, Masaru},
   author={Inui, Takahisa},
   author={Nishimura, Kuranosuke},
   title={Scattering for the quadratic nonlinear Schr\"odinger system in
   ${\bf R}^5$ without mass-resonance condition},
   journal={Funkcial. Ekvac.},
   volume={64},
   date={2021},
   number={3},
   pages={261--291},
   issn={0532-8721},
   review={\MR{4360610}},
   doi={10.1619/fesi.64.261},
}

\bib{HaMa}{article}{
   author={Hamano, Masaru},
   author={Masaki, Satoshi},
   title={A sharp scattering threshold level for mass-subcritical nonlinear
   Schr\"odinger system},
   journal={Discrete Contin. Dyn. Syst.},
   volume={41},
   date={2021},
   number={3},
   pages={1415--1447},
   issn={1078-0947},
   review={\MR{4201846}},
   doi={10.3934/dcds.2020323},
}

\bib{HLN2011}{article}{
   author={Hayashi, Nakao},
   author={Li, Chunhua},
   author={Naumkin, Pavel I.},
   title={On a system of nonlinear Schr\"odinger equations in 2D},
   journal={Differential Integral Equations},
   volume={24},
   date={2011},
   number={5-6},
   pages={417--434},
   issn={0893-4983},
   review={\MR{2809614}},
}

\bib{HLN2012}{article}{
   author={Hayashi, Nakao},
   author={Li, Chunhua},
   author={Naumkin, Pavel I.},
   title={Modified wave operator for a system of nonlinear Schr\"odinger
   equations in 2d},
   journal={Comm. Partial Differential Equations},
   volume={37},
   date={2012},
   number={6},
   pages={947--968},
   issn={0360-5302},
   review={\MR{2924463}},
   doi={10.1080/03605302.2012.668256},
}

\bib{HLN2016}{article}{
   author={Hayashi, Nakao},
   author={Li, Chunhua},
   author={Naumkin, Pavel I.},
   title={Nonlinear Schr\"odinger systems in 2d with nondecaying final data},
   journal={J. Differential Equations},
   volume={260},
   date={2016},
   number={2},
   pages={1472--1495},
   issn={0022-0396},
   review={\MR{3419736}},
   doi={10.1016/j.jde.2015.09.033},
}

\bib{HaNa}{article}{
   author={Hayashi, Nakao},
   author={Naumkin, Pavel I.},
   title={Asymptotics for large time of solutions to the nonlinear
   Schr\"{o}dinger and Hartree equations},
   journal={Amer. J. Math.},
   volume={120},
   date={1998},
   number={2},
   pages={369--389},
   issn={0002-9327},
   review={\MR{1613646}},
}

\bib{HaNa2013}{article}{
   author={Hayashi, Nakao},
   author={Naumkin, Pavel I.},
   title={A system of quadratic nonlinear Klein-Gordon equations in 2d},
   journal={J. Differential Equations},
   volume={254},
   date={2013},
   number={8},
   pages={3615--3646},
   issn={0022-0396},
   review={\MR{3020890}},
   doi={10.1016/j.jde.2013.01.035},
}



\bib{HOT}{article}{
   author={Hayashi, Nakao},
   author={Ozawa, Tohru},
   author={Tanaka, Kazunaga},
   title={On a system of nonlinear Schr\"odinger equations with quadratic
   interaction},
   journal={Ann. Inst. H. Poincar\'e{} C Anal. Non Lin\'eaire},
   volume={30},
   date={2013},
   number={4},
   pages={661--690},
   issn={0294-1449},
   review={\MR{3082479}},
   doi={10.1016/j.anihpc.2012.10.007},
}

\bib{IKN2019}{article}{
   author={Inui, Takahisa},
   author={Kishimoto, Nobu},
   author={Nishimura, Kuranosuke},
   title={Scattering for a mass critical NLS system below the ground state
   with and without mass-resonance condition},
   journal={Discrete Contin. Dyn. Syst.},
   volume={39},
   date={2019},
   number={11},
   pages={6299--6353},
   issn={1078-0947},
   review={\MR{4026982}},
   doi={10.3934/dcds.2019275},
}

\bib{IKN2020}{article}{
   author={Inui, Takahisa},
   author={Kishimoto, Nobu},
   author={Nishimura, Kuranosuke},
   title={Blow-up of the radially symmetric solutions for the quadratic
   nonlinear Schr\"odinger system without mass-resonance},
   journal={Nonlinear Anal.},
   volume={198},
   date={2020},
   pages={111895, 10},
   issn={0362-546X},
   review={\MR{4090442}},
   doi={10.1016/j.na.2020.111895},
}

\bib{IOU}{article}{
   author={Iwabuchi, Tsukasa},
   author={Ogawa, Takayoshi},
   author={Uriya, Kota},
   title={Ill-posedness for a system of quadratic nonlinear Schr\"odinger
   equations in two dimensions},
   journal={J. Funct. Anal.},
   volume={271},
   date={2016},
   number={1},
   pages={136--163},
   issn={0022-1236},
   review={\MR{3494245}},
   doi={10.1016/j.jfa.2016.04.017},
}


\bib{KaSa}{article}{
   author={Katayama, Soichiro},
   author={Sakoda, Daisuke},
   title={Asymptotic behavior for a class of derivative nonlinear
   Schr\"{o}dinger systems},
   journal={Partial Differ. Equ. Appl.},
   volume={1},
   date={2020},
   number={3},
   pages={Paper No. 12, 41},
   issn={2662-2963},
   review={\MR{4336288}},
   doi={10.1007/s42985-020-00012-4},
}

\bib{KaMa}{misc}{
   author={Katayama, Soichiro},
   author={Masaki, Satoshi},
   title={Asymptotic analysis without structural condition for systems of critical derivative nonlinear Schr\"odinger equations},
   status={preprint}
}


%



%





\bib{Ma2}{article}{
   author={Masaki, Satoshi},
   title={Partial classification of the large-time behavior of solutions to
   cubic nonlinear Schr\"odinger systems},
   journal={Partial Differ. Equ. Appl.},
   volume={6},
   date={2025},
   number={4},
   pages={Paper No. 27, 52},
   issn={2662-2963},
   review={\MR{4917200}},
   doi={10.1007/s42985-025-00334-1},
}
\bib{Ma3}{misc}{
	author={Masaki, Satoshi},
   title={Global existence and large-time behavior of solutions to cubic nonlinear Schr\"{o}dinger systems without coercive conserved quantity},
   year={2024},
   	status={available as arXiv:2412.00413},
}



%
\bib{MaSeUr2}{article}{
   author={Masaki, Satoshi},
   author={Segata, Jun-Ichi},
   author={Uriya, Kota},
   title={Asymptotic behavior in time of solution to system of cubic
   nonlinear Schr\"odinger equations in one space dimension},
   conference={
      title={Mathematical physics and its interactions},
   },
   book={
      series={Springer Proc. Math. Stat.},
      volume={451},
      publisher={Springer, Singapore},
   },
   isbn={978-981-97-0363-0},
   isbn={978-981-97-0364-7},
   date={[2024] \copyright 2024},
   pages={119--180},
   review={\MR{4775002}},
}

\bib{MaSeUr3}{article}{
  author={Masaki, Satoshi},
   author={Segata, Jun-Ichi},
   author={Uriya, Kota},
	title={Non-polynomial conserved quantities for ODE systems and their application to the long-time behavior of solutions to cubic NLS systems},
    journal={to appear in Trans. Amer. Math. Soc.},
}

%




\bib{NoPa}{article}{
   author={Noguera, Norman},
   author={Pastor, Ademir},
   title={On the dynamics of a quadratic Schr\"odinger system in dimension
   $n = 5$},
   journal={Dyn. Partial Differ. Equ.},
   volume={17},
   date={2020},
   number={1},
   pages={1--17},
   issn={1548-159X},
   review={\MR{4071981}},
   doi={10.4310/DPDE.2020.v17.n1.a1},
}

\bib{OU2013}{article}{
   author={Ogawa, Takayoshi},
   author={Uriya, Kota},
   title={Asymptotic behavior of solutions to a quadratic nonlinear
   Schr\"odinger system with mass resonance},
   conference={
      title={Harmonic analysis and nonlinear partial differential equations},
   },
   book={
      series={RIMS K\^oky\^uroku Bessatsu},
      volume={B42},
      publisher={Res. Inst. Math. Sci. (RIMS), Kyoto},
   },
   date={2013},
   pages={153--170},
   review={\MR{3220153}},
}

\bib{OU2015}{article}{
   author={Ogawa, Takayoshi},
   author={Uriya, Kota},
   title={Final state problem for a quadratic nonlinear Schr\"odinger system
   in two space dimensions with mass resonance},
   journal={J. Differential Equations},
   volume={258},
   date={2015},
   number={2},
   pages={483--503},
   issn={0022-0396},
   review={\MR{3274766}},
   doi={10.1016/j.jde.2014.09.022},
}
\bib{Oz}{article}{
   author={Ozawa, Tohru},
   title={Long range scattering for nonlinear Schr\"{o}dinger equations in one
   space dimension},
   journal={Comm. Math. Phys.},
   volume={139},
   date={1991},
   number={3},
   pages={479--493},
   issn={0010-3616},
   review={\MR{1121130}},
}
	



%
%
%
%


\bib{Ur16}{article}{
   author={Uriya, Kota},
   title={Final state problem for a system of nonlinear Schr\"odinger
   equations with three wave interaction},
   journal={J. Evol. Equ.},
   volume={16},
   date={2016},
   number={1},
   pages={173--191},
   issn={1424-3199},
   review={\MR{3466217}},
   doi={10.1007/s00028-015-0297-z},
}

\bib{Ur}{article}{
   author={Uriya, Kota},
   title={Final state problem for systems of cubic nonlinear Schr\"{o}dinger
   equations in one dimension},
   journal={Ann. Henri Poincar\'{e}},
   volume={18},
   date={2017},
   number={7},
   pages={2523--2542},
   issn={1424-0637},
   review={\MR{3665222}},
}



\end{biblist}
\end{bibdiv}
\end{document}